\documentclass[runningheads]{llncs}

\usepackage[T1]{fontenc}
\usepackage{graphicx}
\usepackage[caption=false]{subfig}
\usepackage{wrapfig}
\usepackage{amsmath,amsfonts,bm}

\usepackage{hyperref}
\usepackage{color}

\usepackage{algorithm}
\usepackage[commentColor=black]{algpseudocodex}
\usepackage{amssymb}
\usepackage{booktabs}

\newcommand{\figleft}{{\em (Left)}}

\newcommand{\figright}{{\em (Right)}}

\def\Figref#1{Fig.~\ref{#1}}

\def\eqref#1{equation~\ref{#1}}
\def\Eqref#1{Equation~\ref{#1}}
\def\Twoeqref#1#2{Equations \ref{#1} and \ref{#2}}

\def\Algref#1{Algorithm~\ref{#1}}

\def\1{\bm{1}}

\def\vzero{{\mathbf{0}}}

\def\vp{{\mathbf{p}}}

\def\vs{{\mathbf{s}}}

\def\vv{{\mathbf{v}}}

\def\vx{{\mathbf{x}}}
\def\vy{{\mathbf{y}}}

\DeclareMathAlphabet{\mathsfit}{\encodingdefault}{\sfdefault}{m}{sl}
\SetMathAlphabet{\mathsfit}{bold}{\encodingdefault}{\sfdefault}{bx}{n}

\newcommand{\normltwo}{L^2}
\newcommand{\normlp}{L^p}

\DeclareMathOperator*{\argmin}{arg\,min}

\newcommand{\TODO}[1][]{\textcolor{red}{\bf TODO}} 
\begin{document}
\title{CGD: Modifying the Loss Landscape by Gradient Regularization}

\author{
    Shikhar Saxena \and %
    Tejas Bodas \and %
    Arti Yardi
}
\authorrunning{S. Saxena et al.}
\institute{International Institute of Information Technology (IIIT) Hyderabad, India \\
\email{\{shikhar.saxena@research.,tejas.bodas@,arti.yardi@\}iiit.ac.in}}

\maketitle              %

\begin{abstract}
Line-search methods are commonly used to solve optimization problems. The simplest line search method is steepest descent where one always moves in the direction of the negative gradient. Newton's method on the other hand is a second-order method that uses the curvature information in the Hessian to pick the descent direction. In this work, we propose a new line-search method called Constrained Gradient Descent (CGD) that implicitly changes the landscape of the objective function for efficient optimization. CGD is formulated as a solution to the constrained version of the original problem where the constraint is on a function of the gradient. We optimize the corresponding Lagrangian function thereby favourably changing the landscape of the objective function. This results in a line search procedure where the Lagrangian penalty acts as a control over the descent direction and can therefore be used to iterate over points that have smaller gradient values, compared to iterates of vanilla steepest descent. We establish global linear convergence rates for CGD and provide numerical experiments on synthetic test functions to illustrate the performance of CGD. We also provide two practical variants of CGD, CGD-FD which is a Hessian free variant and CGD-QN, a quasi-Newton variant and demonstrate their effectiveness.

\keywords{Numerical Optimization \and Gradient Regularization.}
\end{abstract}
\section{Introduction}

\begin{figure}[th]
    \centering
    \includegraphics[
        width=0.42\textwidth,
        alt={3D surface plots showing the "original loss function f₁(x)" (a blue, shallow curve) and the "steeper loss function f₂(x)" (a brown, steeper curve).}
    ]{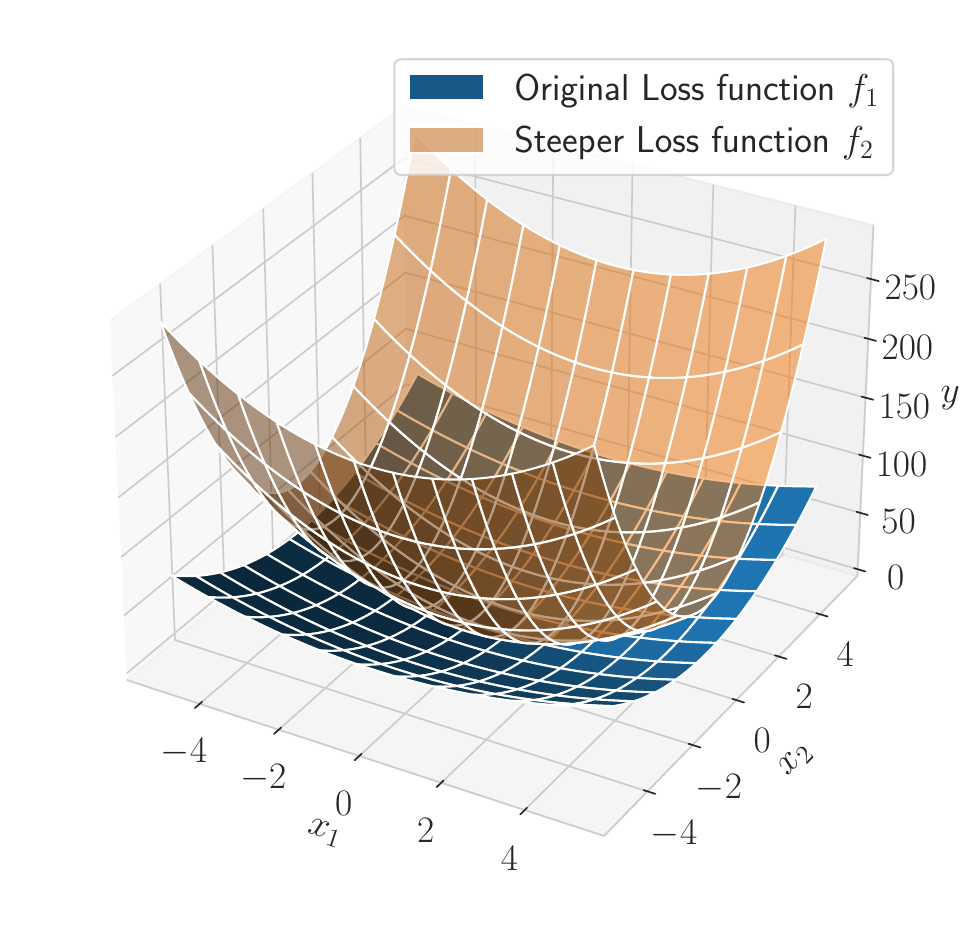}
    \includegraphics[
        width=0.52\textwidth,
        alt={A contour plot and gradient norm heatmap of the original loss function f₁(x) overlaid with two 5-step GD trajectories, one over f₁ and the other over f₂. Both start from the same point, but the path over f₂ reaches the local minimum more directly.}
    ]{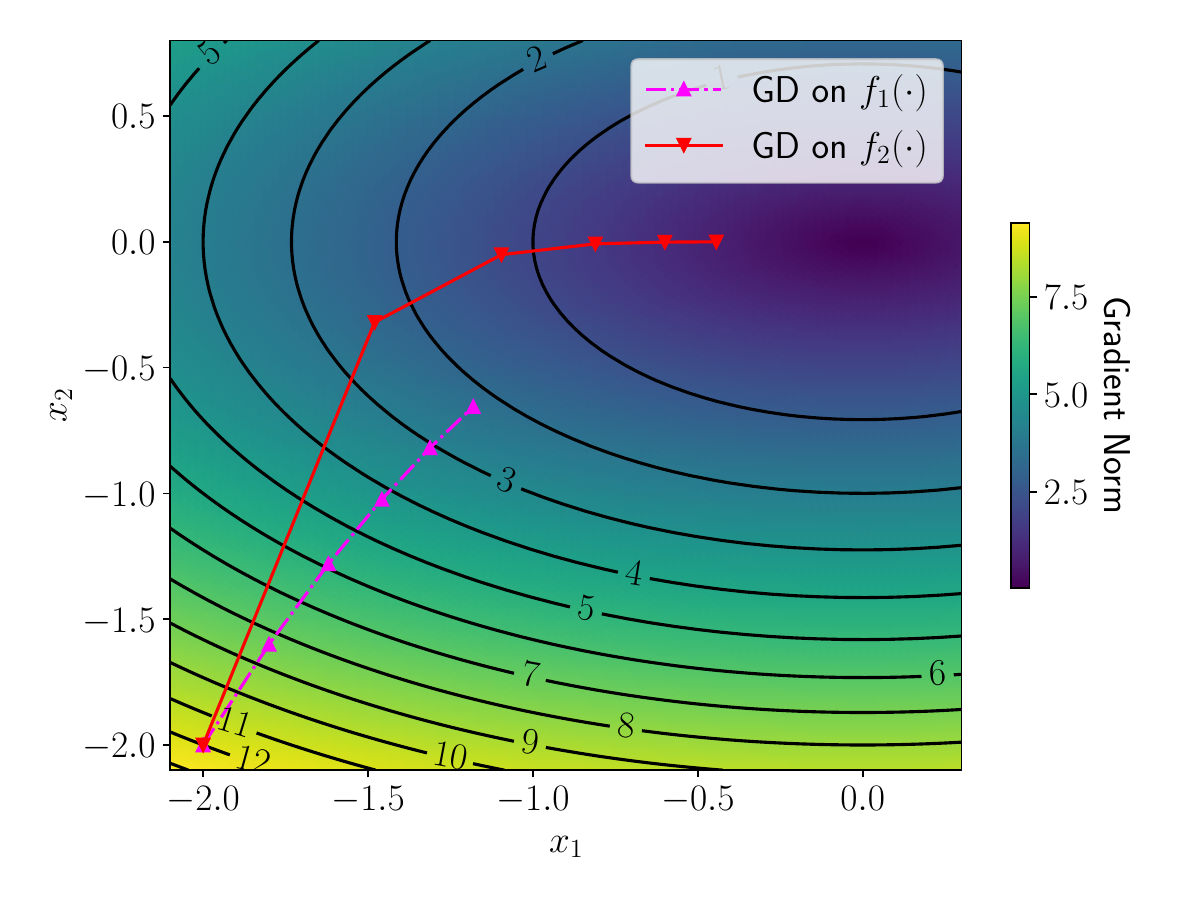}
    \caption{\figleft{} Loss function $f_1(\vx) = x_1^2 + 2 x_2^2$ and a steeper loss function $f_2(\vx)$.\\ \figright{} 5 steps of GD on functions $f_1(\vx)$ and $f_2(\vx)$ with a fixed step-size $\alpha =0.05$. The contours and the gradient norm heatmap are over the function $f_1(\vx)$.}
    \label{fig:surfaces}
\end{figure}

Line-search methods such as gradient descent and Newton's method have been extremely popular in solving unconstrained optimization problems \cite{ruder2017overviewgradientdescentoptimization}. 
However, vanilla version of such algorithms typically produce unsatisfactory results when applied to large scale optimization problems such as those involving deep neural networks. In such problems, the objective/loss function has a very non-linear landscape that is embedded with several local maxima and minima making them difficult to optimize \cite{DBLP:journals/corr/abs-1712-09913}. 
To tackle this problem, one of the approaches used is to modify the objective/loss function in such a way that the new function is easier to optimize. This is typically done by adding a regularization term that will make the loss landscape more smooth, and which would in turn reduce the chances of the optimizer reaching sub-optimal minima, thereby improving generalizability \cite{sam,du2022sharpness,zhuang2022surrogate,implicit-grad,pmlr-v162-zhao22i,pmlr-v202-karakida23a}. 

In this work, we focus on this idea of modifying the objective function for effective numerical optimization and propose an algorithmic paradigm to achieve it. To illustrate the key idea, consider minimizing a loss function $f_1(\cdot)$ over $\mathbb{R}^2$ as shown in \Figref{fig:surfaces} \figleft{}. Additionally consider a steeper loss function $f_2(\cdot)$ such that for every $\vx \in \mathbb{R}^2$ we have $f_1(\vx) \leq f_2(\vx)$. Note that the functions are such that their minima coincides. Now perform gradient descent (GD) with a fixed step-size on both these functions. Their respective trajectories are plotted in \Figref{fig:surfaces} \figright{} on the contours of $f_1$. It is clear from the figure that the iterates of GD on $f_2$ are much closer to the minima as compared to iterates of GD on $f_1$. We make this observation a focal point of this work and investigate if one can achieve the latter iterates (iterates from GD on $f_2$) on the former ($f_1$).

Towards this, we propose Constrained gradient descent (CGD), a variant of the GD algorithm which achieves this by constraining the gradient norm to be appropriately small. Instead of solving this constrained optimization problem, we consider the Lagrangian of this function and perform GD on it. The Lagrangian parameter $\lambda$ controls the penalty on gradient norm and thereby controls the steepness of the modified function. It is on this modified function that we seek to apply GD to possibly attain iterates that are much closer to the local minima as compared to GD iterates. From \Figref{fig:surfaces} \figright{} we additionally see that compared to GD, the CGD iterates (which is nothing but GD iterates on $f_2$) have a lower gradient norm. Since the path to minima is over points with lower gradient norm, this we believe is an attractive feature and even amounts to better generalization properties for high dimensional functions such as loss landscapes of neural networks. 
The idea of flat minima and their attractiveness goes long back to Hochreiter and Schmidhuber \cite{10.1162/neco.1997.9.1.1}: An optimal minimum is considered ``flat'' if the test error changes less in its neighbourhood. Keskar et al. \cite{Keskar_Mudigere_Nocedal_Smelyanskiy_Tang_2017} and Chaudhari et al. \cite{Chaudhari_2017} observe better generalization results for neural networks at flat minima. 
Gradient regularization has only been recently explored, that too from a numerical perspective for deep neural networks (DNNs) where it has been shown that fixed learning rates result in implicit gradient regularization ~\cite{pmlr-v162-zhao22i,pmlr-v202-karakida23a,implicit-grad,origin-of-implicit}.

A key aim of this work is to understand gradient regularization from the perspective of a numerical optimization algorithm and identify properties and features that may not have been obvious earlier. We believe CGD and its variants discussed in this work have the ability to find flatter minima, something which is useful while training neural networks.  
We summarize our contributions below:
\begin{itemize}
\item We propose a new line-search procedure called Constrained Gradient Descent (CGD) that performs gradient descent on a gradient regularized loss function. We also provide global linear convergence for CGD in Theorem \ref{thm:conv}. 
\item While CGD requires the Hessian information, we also propose a first-order variant of CGD using finite-difference approximation of the Hessian called \textit{CGD-FD}. We define appropriate stopping criteria in CGD-FD for settings where gradient computation can be expensive.
\item We conduct experiments over synthetic test functions to compare the performance of CGD and its variants compared to standard line-search procedures.
\item Our work also provides new insights to existing gradient regularization based methods. In fact, we re-interpret and identify pitfalls in the Explicit Gradient Regularization (EGR) Method \cite{implicit-grad} using our formulation.
\end{itemize}

The remainder of the paper is organized as follows. In the next section, we recall some preliminaries on line-search methods. We then discuss the CGD algorithm and propose its variants. We then illustrate the performance of our algorithm on several test functions and conclude with a discussion on future directions.     

\section{Notation and Preliminaries}

The set of real numbers and non-negative real numbers is denoted by $\mathbb{R}$ and $\mathbb{R}^{+}$ respectively.
We consider a vector $\mathbf{x} \in \mathbb{R}^n$ as a column vector given by $\mathbf{x} = \begin{bmatrix} x_1 & x_2 & \ldots & x_n \end{bmatrix}^T$. 
We use a boldface letter to denote a vector and lowercase letters (with subscripts) to denote its components. 
$\normlp$-norm of a vector $\mathbf{x}$ is defined as $\Vert\mathbf{x}\Vert_p = \left(\sum_i |{x_i}|^p\right)^{1/p}$. Setting $p=2$ gives us $\normltwo$-norm: $\Vert\vx\Vert \triangleq \Vert\vx\Vert_2 = \sqrt{\sum_i x_i^2} = \sqrt{\vx^T\vx}$.
$\lambda_{\max}(\cdot)$ and $\lambda_{\min}(\cdot)$ denote the maximum and minimum eigenvalue of a matrix respectively. 
Norm of a matrix is assumed to be the spectral norm \textit{i.e.}, $\|A\| = \sqrt{\lambda_{\max}(A^T A)}$.
$I_n$ denotes the identity matrix of size $n\times n$.
All zero vector of length $n$ is denoted by $\mathbf{0}_n$.

\subsection{Line-Search Methods}

Let $f(\mathbf{x})$ be a twice differential function with domain $\mathcal{D} \subseteq \mathbb{R}^n$ and codomain $\mathbb{R}$, i.e., 
$f:\mathcal{D} \rightarrow \mathbb{R}$. 
Let $\nabla f(\mathbf{x}_k)$ and $H(\mathbf{x}_k)$ denote the gradient and Hessian of the function $f(\mathbf{x})$ evaluated at point $\mathbf{x}_k \in \mathcal{D}$. 
For the sake of simplicity, we will also use the notation
$\nabla f_k$ and $H_k$ to denote $\nabla f(\mathbf{x}_k)$ and $H(\mathbf{x}_k)$ respectively. 
For the given function $f(\mathbf{x})$, 
we consider the problem of finding its minimizer, i.e., we wish to find ${\mathbf{x}^*}\in \mathcal{D}$ such that
\begin{align}
{\mathbf{x}^*} = \argmin_{
\mathbf{x} \in \mathcal{D}} f(\mathbf{x}).
\label{Eqn_x_minimizer_problem}
\end{align}
We focus on the situation where $\vx^*$ is obtained using an iterative line-search procedure (for details refer~\cite[Ch.~3]{Book_Numerical_Optimization}). The steepest descent (\emph{gradient descent}) method, Newton's method, and quasi-Newton's method are some examples of line-search methods. In an iterative algorithm, the key idea is to begin with an initial guess $\mathbf{x}_0$ for the minimizer and generate a sequence of vectors $\mathbf{x}_0, \mathbf{x}_1, \ldots$ until convergence (or until desired level of accuracy is achieved). In such algorithms, $\mathbf{x}_{k+1}$ is obtained using $\mathbf{x}_k$ using a pre-defined update rule. 

For line-search methods, the update rule can be written in a general form as, $
\mathbf{x}_{k+1} = \mathbf{x}_k + \alpha \vp_k
$
where $\alpha \in \mathbb{R}^{+}$ is the step-size (or learning rate) and $\vp_k \in \mathbb{R}^{n}$ corresponds to the direction in the $k^\text{th}$ iteration. 
For $\vp_k$ to be a \emph{descent} direction, the following condition must hold:
\begin{equation}
    \nabla f_k^T \vp_k < 0 \label{eqn:descent_cond}.
\end{equation}

\subsection{Quasi-Newton Methods}
Quasi-Newton methods (QN methods) are line-search methods that maintain an approximation of the inverse of the Hessian to emulate Newton's direction. 
The iterates here are of the form,
\begin{equation}
    \vx_{k+1} = \vx_{k} - \alpha G_k \nabla f_k
\end{equation}

where $G_k$ is a positive definite matrix that is updated at every step to approximate the inverse Hessian. 

We consider DFP and BFGS algorithms that update $G_k$ using symmetric rank-two updates at each descent step \cite{Broyden_1967,Fletcher1970ANA}. 
The update equations for DFP and BFGS are given as follows (for more details refer~\cite[Ch.~6]{Book_Numerical_Optimization}):
\begin{align*}
    \textbf{(DFP) } G_{k + 1} &= G_k + \frac{\vs_k \vs_k^T}{\vy_k^T \vs_k} - \frac{G_k \vy_k \vy_k^T G_k}{\vy_k^T G_k \vy_k}\\
    \textbf{(BFGS) } G_{k + 1} &= \left(I_n - \frac{\vs_k \vy_k^T}{\vy_k^T \vs_k}\right) G_k \left(I_n - \frac{\vy_k \vs_k^T}{\vy_k^T \vs_k}\right) + \frac{\vs_k \vs_k^T}{\vy_k^T \vs_k}
\end{align*}
where $\vs_k = \vx_{k + 1} - \vx_k$ and $\vy_k = \nabla f_{k + 1} - \nabla f_k$.

We can derive the corresponding Hessian approximation $\tilde{G}_k$ from the update steps for $G_k$ by applying the Sherman-Morrison-Woodbury formula~\cite[App. A]{Book_Numerical_Optimization}.
These are given as follows:
\begin{align}
    \textbf{(DFP) } \tilde{G}_{k + 1} &= \left(I_n - \frac{\vy_k \vs_k^T}{\vy_k^T \vs_k}\right) \tilde{G}_k \left(I_n - \frac{\vs_k \vy_k^T}{\vy_k^T \vs_k}\right) + \frac{\vy_k \vy_k^T}{\vy_k^T \vs_k}\label{eqn:dfp_update} \\
    \textbf{(BFGS) } \tilde{G}_{k + 1} &= \tilde{G}_k + \frac{\vy_k \vy_k^T}{\vy_k^T \vs_k} - \frac{\tilde{G}_k \vs_k \vs_k^T \tilde{G}_k}{\vs_k^T \tilde{G}_k \vs_k} \label{eqn:bfgs_update}
\end{align}

\section{Constrained Gradient Descent}\label{sec:cgd}

In this section we propose \textit{Constrained Gradient Descent} (CGD), an iterative line-search method to find a minimizer $\mathbf{x}^{*}$ of the given function $f(\mathbf{x})$ (see \Eqref{Eqn_x_minimizer_problem}). 
The key idea in our approach is to focus on the set of $\mathbf{x} \in \mathcal{D}$ such that $\nabla f(\mathbf{x})$ is close to zero.
For this consider a general constrained optimization problem:
\begin{equation}
    \vx^* = \argmin_{\substack{\mathbf{x} \in \mathcal{D}\\ h\left(\nabla f(\vx)\right) \leq \epsilon}} f(\vx)
    \label{Eqn_Our_approach}
\end{equation}
where the constraint $h(\cdot)$ is defined on the gradient and $\epsilon$ is a small positive real number.
The unconstrained optimization problem corresponding to \Eqref{Eqn_Our_approach} is obtained by penalizing the constraint with a Lagrange multiplier $\lambda > 0$:
\begin{equation} 
\mathbf{x}^{\star} = \argmin_{\mathbf{x} \in \mathcal{D}} \Big[ f(\mathbf{x}) + \lambda \ h \left(\nabla f(\vx)\right) \Big]
\label{eqn:formula}
\end{equation}
Note that $\epsilon$ doesn't affect the optimization and hence ignored from the objective. Also, observe that such a formulation provides us with a \textit{modified} loss function to optimize over. 
A suitable constraint $h(\cdot)$ will make the objective steeper while also keeping the stationary points intact. We call this \textit{penalization} or using a \textit{gradient penalty} since the objective is penalized at points based on its gradient value.
For CGD, we consider $h(\cdot)$ to be the square of $\normltwo$-norm of the gradient.
\begin{equation}
\mathbf{x}^{\star} = \argmin_{\mathbf{x} \in \mathcal{D} } g(\mathbf{x})
\triangleq  f(\mathbf{x}) + \lambda \Vert \nabla f(\mathbf{x}) \Vert^2 = f(\vx) + \lambda \left(\nabla f(\vx)^T \nabla f(\vx)\right)
\label{Eqn_CGD_formulation}
\end{equation}
Steepest descent iterates over $g(\cdot)$ in \Eqref{Eqn_CGD_formulation} are given as 
\begin{align}
\mathbf{x}_{k+1} &= \mathbf{x}_{k} - \alpha \nabla g_k \nonumber \\
&=\mathbf{x}_{k} - \alpha \left( \nabla f_k + 2\lambda  H_k  \nabla f_k \right) \label{eqn:grad_expr}\\
&= \mathbf{x}_{k} - \alpha B_k \nabla f_k \nonumber 
\end{align}
where $B_k \triangleq  I_n + 2\lambda H_k$. Thus, $\vp_k = - B_k \nabla f_k$ 
which is the direction taken at the $k^\text{th}$ iteration by CGD. 
Revisiting the example in \Figref{fig:surfaces}, the function $f_1(\vx) = x_1^2 + 2x_2^2$ was penalized with the square of $\normltwo$-norm of gradient as given in \Eqref{Eqn_CGD_formulation}.
Thus, the modified loss function $f_2(\vx) = x_1^2 + 2x_2^2 + \lambda \left((2x_1)^2 + (4 x_2)^2\right) = (1 + 4\lambda)x_1^2 + 2(1+8\lambda)x_2^2$ where $\lambda$ was chosen to be $0.4$.

We now provide a lemma that investigates if the penalized objective function has stationary points which are different from the original function and if so characterizes them.  
\begin{lemma} \label{lm:cond_asp}
Let $\mathcal{S}_{\hat{\mathbf{x}}}$ and $\mathcal{S}_{\mathbf{x}^{\star}}$ be the stationary points of $f(\vx)$ and $g(\vx)$ as defined in  \Eqref{Eqn_CGD_formulation}. Then, $\mathcal{S}_{\hat{\vx}} \subseteq \mathcal{S}_{\vx^*}$. Furthermore, for any $\vx^* \in \mathcal{S}_{\vx^*}$ one of the following is true:
     (a) $\vx^* \in \mathcal{S}_{\hat{\vx}}$ or
     (b) $\nabla f(\vx^*)$ is an eigenvector of $H(\vx^*)$ with the eigenvalue $-\frac{1}{2\lambda}$.
\end{lemma}
\begin{proof}
    From \Eqref{eqn:grad_expr}, for any $\hat{\mathbf{x}}\in \mathcal{S}_{\hat{\mathbf{x}}}, \ \nabla f(\hat{\mathbf{x}}) = \mathbf{0}_n \Rightarrow \nabla g(\hat{\vx}) = \mathbf{0}_n$.
    So, $\mathcal{S}_{\hat{\mathbf{x}}}\subseteq \mathcal{S}_{\mathbf{x}^{*}}$ holds trivially.
    \noindent{}
    Now for any $\vx^* \in \mathcal{S}_{\vx^*}$,
    \(\nabla g(\vx^*) = \vzero_n \Rightarrow (I_n + 2 \lambda H(\vx^*)) \nabla f(\vx^*) = \vzero_n\).
 If $\nabla f(\vx^*)\not = \vzero_n$, we have $H(\vx^*) \nabla f(\vx^*) = - \frac{1}{2\lambda} \nabla f(\vx^*)$ which proves the result. \qed
\end{proof}
The above lemma illustrates that additional stationary points could possibly be introduced and some of these points could also be local minima. The lemma also characterizes conditions under which this is true and therefore such points can easily be detected. A perturbation from the current $\lambda$ in that case results in the iterate to descend further, possibly moving towards a better minimum.

\begin{figure}[!htb]
    \centering
    \includegraphics[width=0.65\linewidth,
    alt={Two curves plotted over a shared x-axis: a blue curve labeled "Loss Landscape f(x)" and an orange curve labeled "Penalized Landscape g(x)". The penalized curve introduces more stationary points than the original. Black arrows mark the negative gradient directions on both curves, indicating the change in descent direction. Gradients are shown at the point x = -7.7, which is highlighted by a dashed vertical line.}
    ]{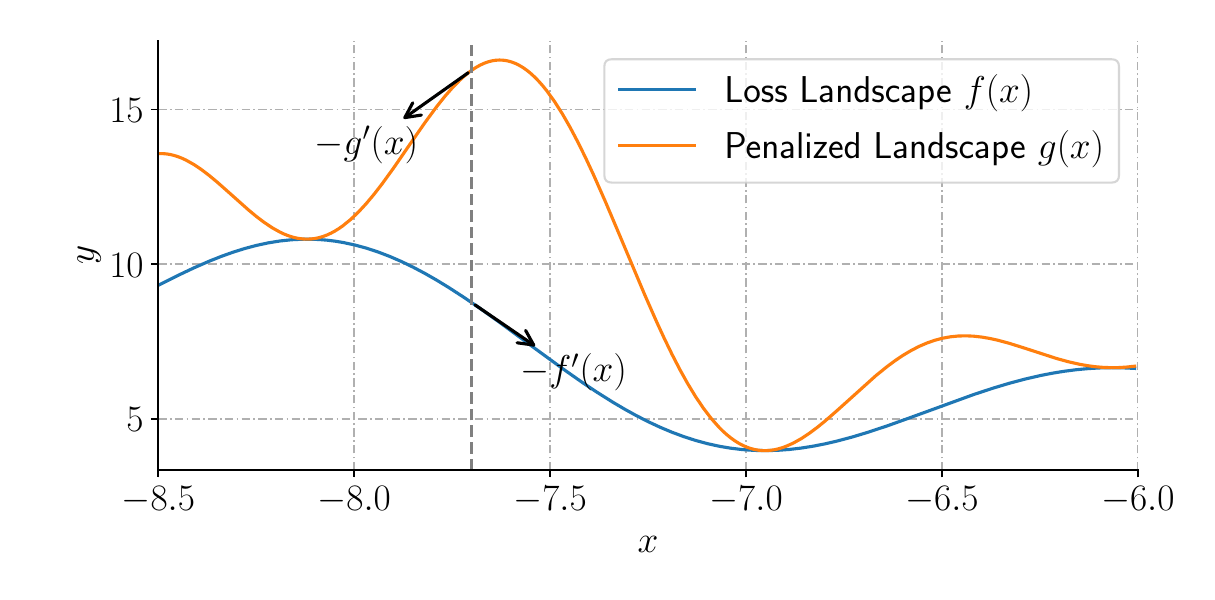}
    \caption{Penalization introducing new stationary points at $x\approx -6.5$ and $x\approx -7.6$, and changing the local maximum at $x\approx -8.1$ to a local minimum.}
    \label{fig:ascent}
\end{figure}

Further note that within the current scheme, the nature of stationary points of $f(\vx)$ might change in $g(\vx)$. Particularly, the local maxima and saddle points of $f(\vx)$ might become local minima in $g(\vx)$ depending on $\lambda$. This happens because penalization does not change the function value at the stationary point while spiking up the function values at points in the neighbourhood.

As a result, descent on the penalized function might actually cause ascent over the original loss function. For example in \Figref{fig:ascent}, observe that at $x=-7.7$ (marked with the dashed line), steepest descent along $- \nabla g(\vx)$ direction would actually cause an ascent over the original loss function.

To fix this behaviour, we ensure that we only move along the direction $-\nabla g(\vx)$ when it is a descent direction. Therefore, we only move along $-\nabla g(\vx)$, if $\nabla f(\vx)^T \nabla g(\vx) > 0$ (from \Eqref{eqn:descent_cond}). Otherwise, we set $\lambda = 0$ at this point and move along the $-\nabla f(\vx)$ direction. Note that this simple check also helps us to avoid stopping at artificially introduced stationary points. 
We summarize CGD in \Algref{algo:cgd}.

\begin{algorithm}[!htb]
    \caption{Constrained Gradient Descent (CGD)}
    \label{algo:cgd}
    \textbf{Input}: Objective function $f: \mathcal{D} \rightarrow \mathbb{R}$, initial point $\vx_0$, max iterations $T$, step-size $\alpha$, regularization coefficient $\lambda$.\\
    \textbf{Output}: Final point $\vx_T$.
    \begin{algorithmic}[1]
    \For{iteration $k = 0,\ldots, T - 1$}
        \State \(\vp_k \gets - \left(I_n + 2 \lambda_k H_{k}\right) \nabla f_{k}\)
        \If{$\nabla f_k^T \vp_k < 0$}  
        \Comment{Check if $\vp_k$ is a Descent Direction}
        \State $\vx_{k + 1} \gets \vx_{k} + \alpha \vp_k$
        \Else
        \State $\vx_{k + 1} \gets \vx_{k} - \alpha \nabla f_k$
        \EndIf
    \EndFor
    \State \Return{$\vx_{T}$}      
    \end{algorithmic}
\end{algorithm}

\subsection{Global Linear Convergence of CGD}
We begin with the following definition followed by a theorem on convergence guarantees for CGD.
\begin{definition}[PL Inequality \cite{polyak_gm}]
    A function $f: \mathcal{D} \rightarrow \mathbb{R}$ satisfies the PL inequality if for some $\mu > 0$,
    \begin{equation}
        \frac{1}{2} \Vert \nabla f(x) \Vert^2  \ge \mu(f(x) - f^*),  \ \forall x \in \mathcal{D}.
        \label{eqn:pl_ineq_def}
    \end{equation}
where $f^*$ is the optimal function value.
\end{definition}

\begin{theorem}\label{thm:conv}
Let $f$ be a convex, $L$-smooth function on $\mathcal{D} \subseteq \mathbb{R}^n$. Let $\{\vx_k\}_{k \ge 0}$ be the sequence generated by the CGD method as described in \Eqref{eqn:grad_expr}.
Then for a constant step-size \(\alpha \in \left( 0,\frac{2}{{L(1 + 2\lambda L)}^{2}} \right)\), the following holds true \cite[Theorem~4.25]{beck2014introduction}:
\begin{enumerate}
    \item The sequence $\{f_k\}_{k \ge 0}$ is nonincreasing. In addition, for any $k \ge 0, f_{k + 1} < f_k$ unless $\nabla f_k = \vzero$.
    \item $\nabla f_k \rightarrow 0$ as $k \rightarrow \infty$.
\end{enumerate}
Furthermore, if $f$ satisfies the PL inequality (\Eqref{eqn:pl_ineq_def}) then the CGD method with a step-size of $\frac{1}{{L(1 + 2\lambda L)}^{2}}$, has a global linear convergence rate,
\[
f_k - f^* \le \left(1 - \frac{\mu}{{L(1 + 2\lambda L)}^{2}} \right)^k (f_0 - f^*).
\]
\end{theorem}

\begin{proof}
    Since $f$ is convex and $L$-smooth we have,
\begin{itemize}
\item 
  \(\mathbf{0} \preceq H \preceq L{I} \Leftrightarrow 0 \leq \mathbf{v}^{T}H\mathbf{v} \leq L\left\| \mathbf{v} \right\|^{2} \Leftrightarrow 0 \leq \left\| H \right\| \leq L\)
\item
  \(f\left( \mathbf{y} \right) \leq f\left( \mathbf{x} \right) + \nabla{f\left( \mathbf{x} \right)}^{T}\left( \mathbf{y} - \mathbf{x} \right) + \frac{L}{2}\left\| {\mathbf{y} - \mathbf{x}} \right\|^{2}\)
\end{itemize}
    Substituting \(\mathbf{x}_{k + 1}\text{ and }\mathbf{x}_{k}\) in place of
\(\mathbf{y}\text{ and }\mathbf{x}\). And
\(\mathbf{x}_{k + 1} - \mathbf{x}_{k} = - \alpha\nabla g_{k}\),
\begin{equation}
f_{k + 1} \leq f_{k} - \alpha\nabla f_{k}^{T}\nabla g_{k} + \frac{L}{2}\alpha^{2}\left\| {\nabla g_{k}} \right\|^{2} \label{eqn:l_sm_expr}
\end{equation}
Bounds on $\Vert\nabla g_k\Vert^2$ and $\nabla f_k^T \nabla g_k$ are found as follows: 
\begin{align*}
\left\| {\nabla g_{k}} \right\|^{2} & = \left\| {\nabla f_{k} + 2\lambda H_{k}\nabla f_{k}} \right\|^{2} \\
 & = \left\| {\nabla f_{k}} \right\|^{2} + 4\lambda\nabla f_{k}^{T}H_{k}\nabla f_{k} + 4\lambda^{2}\left\| {H_{k}\nabla f_{k}} \right\|^{2} \\
 & \leq \left\| {\nabla f_{k}} \right\|^{2} + 4\lambda\nabla f_{k}^{T}H_{k}\nabla f_{k} + 4\lambda^{2}\left\| H_{k} \right\|^{2}\left\| {\nabla f_{k}} \right\|^{2} \\
 & \leq \left\| {\nabla f_{k}} \right\|^{2}\left( 1 + 4\lambda L + 4\lambda^{2}L^{2} \right) = (1 + 2\lambda L)^{2}\left\| {\nabla f_{k}} \right\|^{2}
\end{align*}
and \(\nabla f_{k}^{T}\nabla g_{k} = \left\| {\nabla f_{k}} \right\|^{2} + 2\lambda\nabla f_{k}^{T}H_{k}\nabla f_{k} \ge \left\| {\nabla f_{k}} \right\|^{2}\).
Substituting in \Eqref{eqn:l_sm_expr}, 
\begin{align}
f_{k + 1} - f_{k} & \leq - \alpha\left\| {\nabla f_{k}} \right\|^{2} + \frac{L}{2}\alpha^{2}(1 + 2\lambda L)^{2}\left\| {\nabla f_{k}} \right\|^{2} \nonumber \\
 & = - \alpha\left( 1 - \frac{{L(1 + 2\lambda L)}^{2}\alpha}{2} \right)\left\| {\nabla f_{k}} \right\|^{2} \label{eqn:my_m}
\end{align}

Thus, we have, \(f_k - f_{k+1} \ge M \Vert \nabla f_k \Vert^2 \ge 0\)
where $M =  \alpha\left( 1 - \frac{{L(1 + 2\lambda L)}^{2}\alpha}{2} \right) > 0$ for constant step-size \(\alpha \in \left( 0,\frac{2}{{L(1 + 2\lambda L)}^{2}} \right)\). 
So, the equality $f_{k + 1} = f_{k}$ only holds when $\nabla f_k = \vzero$. Furthermore, since the sequence $\{f_k\}_{k\ge 0}$ is non-increasing and bounded below, it converges. So, $\nabla f_k \rightarrow 0$ as $k \rightarrow \infty$.
This proves the first two parts. For the third part, consider the $\mu$-PL assumption; from \Twoeqref{eqn:pl_ineq_def}{eqn:my_m} at step-size $\alpha =  \frac{1}{{L(1 + 2\lambda L)}^{2}}$, we have
\begin{align*}
f_{k + 1} - f_{k} &\leq  - \frac{1}{2 {L(1 + 2\lambda L)}^{2}} \left\| {\nabla f_{k}} \right\|^{2} \le - \frac{\mu}{{L(1 + 2\lambda L)}^{2}} (f_k - f^*) \\
\text{On subtracting $f^*$}&\text{ from both sides,}\\
\therefore f_{k + 1} - f^* &\le \left(1 - \frac{\mu}{{L(1 + 2\lambda L)}^{2}} \right) (f_k - f^*) 
\end{align*}
Applying this inequality recursively gives us the result.
\qed
\end{proof}

Note that the PL inequality assumption is not necessary for functions that are strongly convex or strictly convex functions (over compact sets) since the inequality is already satisfied. Karimi et al. \cite{karimi2020} further shows that this convergence result can be extended to functions of the form $h(A\vx)$ where $h$ is a strongly (or strictly) convex function composed with a linear function \textit{e.g.}, least-squares problem and logistic regression.
\subsubsection{CGD on Strongly Quadratic Function.} Next we investigate the rate of convergence of CGD on a strongly quadratic function as our objective. Such an analysis helps in identifying how penalization affects the \textit{condition number}, which essentially controls the geometry of descent. For a matrix $A$, the condition number is defined as  $\kappa(A) \triangleq \frac{\lambda_{\max}{(A)}}{\lambda_{\min}(A)}$. A higher condition number or \textit{ill conditioning} implies a more zig-zagging behaviour and thus, slower convergence. 

Let \(f\left( \mathbf{x} \right) = \frac{1}{2}\mathbf{x}^{T}Q\mathbf{x} - \mathbf{b}^{T}\mathbf{x}\) where $Q \succ 0$.
Let $l \triangleq \lambda_{\min}(Q) $ and $L \triangleq \lambda_{\max}(Q)$.
From \Eqref{eqn:grad_expr} we know \(\mathbf{x}_{k + 1} = \mathbf{x}_{k} - \alpha B_{k}\nabla f_{k}\) where \(B_{k} = {I} + 2\lambda Q\) and
\(\nabla f\left( \mathbf{x}_{k} \right) = Q\mathbf{x}_{k} - \mathbf{b}\).
The error at each iterate is then given as:
\begin{align}
\Vert \mathbf{e}_{k + 1} \Vert & = \left\| {\mathbf{x}_{k + 1} - \mathbf{x}^{\ast}} \right\| = \left\| {\mathbf{x}_{k} - \alpha B_{k}\nabla f_{k} - \mathbf{x}^{\ast}} \right\|\nonumber \\
&= \left\| {\mathbf{x}_{k} - \alpha B_{k}\left( Q\mathbf{x}_{k} - \mathbf{b} \right) - \mathbf{x}^{\ast}} \right\| \nonumber \\
 & = \left\| {\left( {I} - \alpha B_{k}Q \right)\mathbf{x}_{k} + \alpha B_{k}\mathbf{b} - \mathbf{x}^{\ast}} \right\| \nonumber \\
 & = \left\| {\left( {I} - \alpha B_{k}Q \right)\mathbf{x}_{k} + \alpha B_{k}\mathbf{b} - \mathbf{x}^{\ast} + \left(\alpha B_{k}Q\mathbf{x}^{\ast} - \alpha B_{k}Q\mathbf{x}^{\ast} \right)}  \right\| \nonumber \\
 & = \left\| {\left( {I} - \alpha B_{k}Q \right)\left( \mathbf{x}_{k} - \mathbf{x}^{\ast} \right) + \alpha B_{k}\left( \mathbf{b} - Q\mathbf{x}^{\ast} \right)} \right\|\qquad \text{(on rearrangement) } \nonumber \\
 & = \left\| {\left( {I} - \alpha B_{k}Q \right)\mathbf{e}_{k}} \right\| \qquad \text{(Gradient at }\mathbf{x}^{\ast} = \mathbf{0}\text{)}\nonumber \\
\therefore \left\| \mathbf{e}_{k+1}\right\|& \leq \left\| {{I} - \alpha B_{k}Q} \right\|\left\| \mathbf{e}_{k} \right\| \label{eqn:eqn_error}
\end{align} 
Since \({I}\) and \(Q\) commute, we have \((1 + 2\lambda l){I} \preccurlyeq B_{k} \preccurlyeq (1 + 2\lambda L){I}\). Also since $Q$ and $B_k$ are both symmetric positive-definite matrices and commute, we can say the following about $B_k Q$, 
\begin{align*}
    &\lambda_{\min}\left( B_{k} \right)\lambda_{\min}(Q) I \preccurlyeq B_k Q \preccurlyeq \lambda_{\max}\left( B_{k} \right)\lambda_{\max}(Q) I \\
    \Leftrightarrow \ &\lambda_{\min}\left( B_{k} \right)\lambda_{\min}(Q) \leq \left\| B_k Q \right\| \leq \lambda_{\max}\left( B_{k} \right)\lambda_{\max}(Q) 
\end{align*} 
\begin{align*}
    \therefore \left\| {{I} - \alpha B_{k}Q} \right\| &\le \max\left\{
    \left| 1 - \alpha\lambda_{ \min } (B_{k})\lambda_{\min}(Q) \right|, 
    \left| 1 - \alpha\lambda_{ \max }(B_{k}) \lambda_{\max}(Q) \right| 
\right\}\\
    &= \max\left\{ \left| 1 - \alpha (1 + 2 \lambda l) l \right|, 
\left| 1 - \alpha (1 + 2 \lambda L) L \right| \right\}
\end{align*}
So in \Eqref{eqn:eqn_error},
\begin{align*}
\left\| {\mathbf{x}_{k + 1} - \mathbf{x}^{\ast}} \right\| & \leq 
\max\left\{ \left| 1 - \alpha (1 + 2 \lambda l) l \right|, 
\left| 1 - \alpha (1 + 2 \lambda L) L \right| \right\}
\left\| {\mathbf{x}_{k} - \mathbf{x}^{\ast}} \right\| \\
 &\le \left( 
\max\left\{ \left| 1 - \alpha (1 + 2 \lambda l) l \right|, 
\left| 1 - \alpha (1 + 2 \lambda L) L \right| \right\}
  \right)^{k}\left\| {\mathbf{x}_{0} - \mathbf{x}^{\ast}} \right\|
\end{align*}

Hence for CGD to converge, 
\[  \max\left\{ \left| 1 - \alpha (1 + 2 \lambda l) l \right|, 
    \left| 1 - \alpha (1 + 2 \lambda L) L \right| \right\} < 1, \ \forall k
\]

A suitable $\alpha$ that satisfies this condition is chosen to be \[\alpha = \frac{2}{\lambda_{\max}( B_{k})\lambda_{\max}(Q) + \lambda_{\min}(B_{k})\lambda_{\min}(Q)} = \frac{2}{(1 + 2\lambda L)L + (1 + 2\lambda l)l}\]
Both terms within the $\max(\cdot)$ achieve the same value at this $\alpha$.
\begin{align*}
\therefore \left\| {\mathbf{x}_{k + 1} - \mathbf{x}^{\ast}} \right\| & \leq \left( \frac{(1 + 2\lambda L)L - (1 + 2\lambda l)l}{(1 + 2\lambda L)L + (1 + 2\lambda l)l} \right)^{k}\left\| {\mathbf{x}_{0} - \mathbf{x}^{\ast}} \right\| \\
&= \left( \frac{\lambda_{\max}(B_k) \lambda_{\max}(Q) - \lambda_{\min}(B_k)\lambda_{\min} (Q)}{\lambda_{\max}(B_k) \lambda_{\max}(Q) + \lambda_{\min}(B_k) \lambda_{\min}(Q)} \right)^{k}\left\| {\mathbf{x}_{0} - \mathbf{x}^{\ast}} \right\|\\
&= \left(\frac{{\kappa}-1}{{\kappa} + 1}\right)^k \|\vx_0 - \vx^*\| 
\end{align*}
where ${\kappa} = \kappa(B_k) \kappa(Q)$.

\begin{remark}
Note that GD has a similar rate with the condition number there being $\kappa(Q)$ instead. And since $\kappa(B_k) > 0$, $\kappa > \kappa(Q)$. Thus, the condition number worsens while descending over the penalized landscape in comparison to the original objective. This indicates that our method would perform poorly for functions that have a high condition number and thus, take longer to converge.
\end{remark}

\subsection{CGD-FD: Finite Difference Approximation of the Hessian}
Note that CGD is a second-order optimization algorithm as it requires the Hessian information to compute each iterate. We improve over this complexity by restricting CGD to be a first-order line search method wherein the Hessian is approximated using a finite difference~\cite{Pearlmutter1994FastEM}. Using Taylor series, we know that \[\nabla f(\vx_k +\Delta \vx) = \nabla f_k + H_k \Delta\vx + \mathcal{O} (\Vert\Delta\vx\Vert^2). \]
Now let $\Delta\vx = r \vv$ where $r$ is arbitrarily small. Then, we can rewrite the above expression as:
\[
H\vv = \frac{\nabla f(\vx_k + r \vv) - \nabla f_k}{r} + \mathcal{O}(\Vert r\Vert).
\]
Substituting $\vv = \nabla f_k$ and using this approximation in \Eqref{eqn:grad_expr} gives us  
\begin{align}
\nabla g(\vx_k) &\approx \nabla f_k + 2 \lambda  \frac{\nabla f(\vx_k + r \nabla f_k) - \nabla f_k}{r} \nonumber\\
&= (1 - \nu) \nabla f_k + \nu \nabla f(\vx_k + r \nabla f_k) \label{eqn:cgd_fd_update}
\end{align}
where $\nu = 2 \lambda/r$. We call this version \textit{CGD with Finite Differences} (CGD-FD).

Note that this requires two gradient calls in every iteration, which might be prohibitive. Towards this it is natural to consider a stopping criterion to revert back to using steepest direction whenever the improvement through CGD-FD iterates is low.  
This is done to avoid the extra gradient evaluations after the optimizer has moved to a sufficiently optimal point in the domain.
Such criteria are particularly helpful in loss functions where gradient evaluations are costly and therefore, budgeted. This also gives us the freedom to play around with which direction to choose (steepest or CGD-FD) and quantify when to switch to the other.
For instance, we initially only move in the directions suggested by CGD-FD and switch to steepest once we make a good-enough drop in the function value from the initial point.

In our experiments, we only use CGD-FD for the first $b$ iterations (out of the total budget $T$) to strive for a good-drop in function values initially. While doing so, we also ensure that we only move in the direction $\vp_k$ (CGD-FD, \Eqref{eqn:cgd_fd_update}) if it is a descent direction. Otherwise, we stop using CGD-FD and use the steepest direction henceforth. 
We summarize our CGD-FD method in \Algref{algo:cgd_fd}.

\begin{algorithm}[!htb]
    \caption{Constrained Gradient Descent using Finite Differences (CGD-FD)}
    \label{algo:cgd_fd}
    \textbf{Input}: Objective function $f: \mathcal{D} \rightarrow \mathbb{R}$, initial point $\vx_0$, max iterations $T$, step-size $\alpha$, regularization coefficients $\lambda$, stopping threshold $b$.\\
    \textbf{Output}: Final point $\vx_T$.
    \begin{algorithmic}[1]
    
    \State $\nu \gets 2 \lambda/ r$, \textsc{use\_cgd} $\gets$ \textsc{true}
    \State Gradient Evaluations $c\gets 0$ 
    \Comment{Can be interpreted as cost}
    \For{$k = 0, \ldots, T - 1$}
        \If{$c = T$} \Return $\vx_k$ \Comment{if budget has been exhausted}
        \EndIf
        \If{\textsc{use\_cgd}}
            \State $\vp_k \gets - (1 - \nu) \nabla f_{k} - \nu \nabla f(\vx_{k} + r \nabla f_{k})$
            \State $c \gets c + 2$
            \If{$\nabla f_k^T \vp_k < 0$} \Comment{Check if $\vp_k$ is a Descent Direction}
                \State $\vx_{k + 1} \gets \vx_{k} + \alpha \vp_k$
                \Else
                \State $\vx_{k + 1} \gets \vx_{k} - \alpha \nabla f_k$
                \State \textsc{use\_cgd} $\gets$ \textsc{false} 
            \EndIf
        \Else
            \State $\vx_{k + 1} \gets \vx_{k} - \alpha \nabla f_k$
            \State $c \gets c + 1$
        \EndIf
        \If{$k \ge b$}
        \State \textsc{use\_cgd} $\gets$ \textsc{false}
        \EndIf
    \EndFor
    \State \Return $\vx_{T}$
    \end{algorithmic}
\end{algorithm}

\subsection{CGD-QN: Quasi-Newton Variants of CGD}

We consider quasi-Newton variants of CGD (CGD-QN) where a positive-definite approximation of the Hessian is maintained. 
The update step for CGD-QN is given as:
\(
\vx_{k+1} = \vx_{k} - \alpha \left(I_n + 2 \lambda_k \tilde{G}_k\right) \nabla f_k
\)
where $\tilde{G}_k$ is the Hessian approximation at step $k$, which is updated after every step. Particularly \Twoeqref{eqn:dfp_update}{eqn:bfgs_update} correspond to the updates of DFP and BFGS variants of CGD respectively. We summarize CGD-QN method in \Algref{algo:cgd_qn}.
\begin{algorithm}[!htb]
    \caption{Constrained Gradient Descent: Quasi Newton (CGD-QN)}
    \label{algo:cgd_qn}

    \textbf{Input}: Objective function $f: \mathcal{D} \rightarrow \mathbb{R}$, initial point $\vx_0$, max iterations $T$, step-size $\alpha$, regularization coefficient $\lambda$.\\
    \textbf{Output}: Final point $\vx_T$.
    \begin{algorithmic}[1]
    \State $\tilde{G}_0 \gets I_n$ 
    \For{iteration $k = 0,\ldots, T - 1$}
        \State \(\vp_k \gets - \left(I_n + 2 \lambda_k \tilde{G}_{k}\right) \nabla f_{k}\)
        \If{$\nabla f_k^T \vp_k < 0$}  
        \Comment{Check if $\vp_k$ is a Descent Direction}
        \State $\vx_{k + 1} \gets \vx_{k} + \alpha \vp_k$
        \Else
        \State $\vx_{k + 1} \gets \vx_{k} - \alpha \nabla f_k$
        \EndIf
        \State Update $\tilde{G}$ according to the QN method used. For CGD-DFP and CGD-BFGS follow the updates in \Twoeqref{eqn:dfp_update}{eqn:bfgs_update} respectively. 
    \EndFor
    \State \Return{$\vx_{T}$}      
    \end{algorithmic}  
\end{algorithm}

\section{Numerical Experiments}

We test CGD-FD and CGD-QN on synthetic test functions from \textit{Virtual Library of Simulation Experiments: Test Functions and Datasets}~\footnote{\url{http://www.sfu.ca/~ssurjano}}.

For our experiments we chose total budget $T=40$ and the stopping threshold $b$ as $T/4$. For the hyperparameter $\lambda$, we empirically tested different schedules and strategies for different kinds of functions. We observed that using constant $\lambda$ works well with convex functions. For some non-convex functions, an increasing schedule is preferred. 
We denote this as $\texttt{L}(a, b)$ which is an increasing linear schedule of $T$ values going from $a$ to $b$. For the step-size $\alpha$, a constant-value was found to be suitable in all our experiments.

To measure the initial drop in function value, we compare the \emph{Improvement} for the first step of CGD-FD vs the first step of steepest descent. Improvement (in \%) is given as \(\frac{f(\vx_0) - f(\vx_1)}{f(\vx_0)} * 100\).

\begin{table}[!htb]
    \centering
    \caption{Initial Improvements (in \%) over test functions (Dimensions=$n$) for choices of $\alpha$ and $\lambda$ with a budget $T=40$ and stopping threshold $b=T/4$.}
    \label{tab:tab-1}
    \begin{tabular}{lcccrrr}
    \toprule
\textbf{Test Function} & $n$ & $\lambda$ & $\alpha$ & \textbf{GD} & \textbf{CGD-FD} \\
    \midrule
Quadratic function & $10$ & 0.4 & 0.01 &
    18.89 & \textbf{97.91}\\
    Rotated hyper-ellipsoid function & $5$ & 0.5 & 0.01 &
    15.94 & \textbf{82.76}\\
    Levy function & $2$ & \texttt{L}(0.01, 0.1) & 0.05 &
    23.73 & \textbf{63.21}\\
    Branin function & $2$ & 0.07 & 0.01  &
    37.53 & \textbf{87.07}\\
    Griewank function & $2$ & 40.0 & 0.01 &
    0.01 & \textbf{0.08}\\
    Matyas function & $2$ & 10.0 & 0.01 &
    1.83 & \textbf{34.40}\\
    \bottomrule
    \end{tabular}
\end{table}

\begin{figure}[!htb]
    \centering
    \subfloat[Quadratic function]{%
        \includegraphics[width=0.324\textwidth,
        alt={Line plot showing drop in function value for CGD-FD vs GD on the Quadratic function. Both methods descend rapidly, with CGD-FD flattening earlier.}]{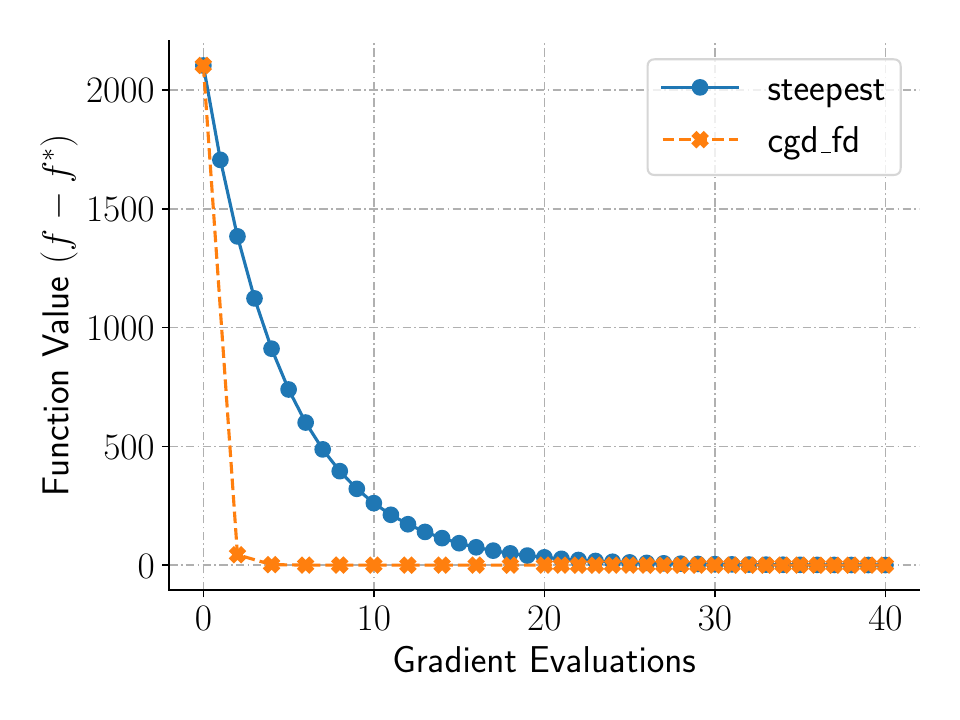}
        \label{fig:qfunc}
    }\hfill
    \subfloat[Rotated hyper-ellipsoid function]{%
        \includegraphics[width=0.324\textwidth,
        alt={Line plot showing drop in function value for CGD-FD vs GD on the Rotated hyper-ellipsoid function. Both methods descend rapidly, with CGD-FD flattening earlier.}]{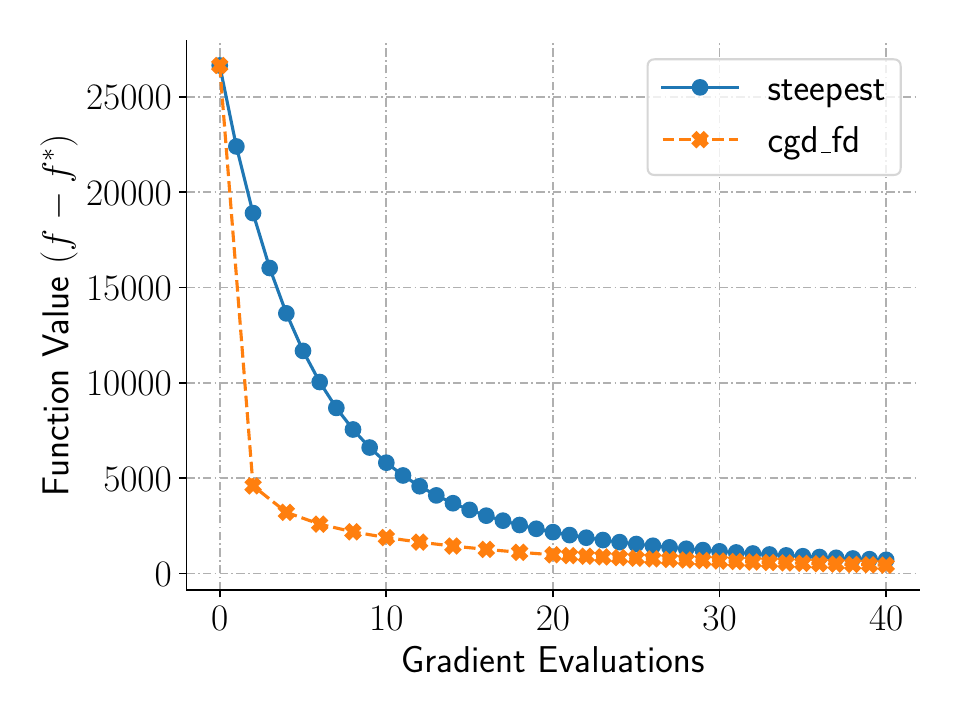}
        \label{fig:rhe}
    }\hfill
    \subfloat[Levy function]{%
        \includegraphics[width=0.324\textwidth,
        alt={Line plot showing drop in function value for CGD-FD vs GD on the Levy function, with CGD-FD showing slightly faster early descent.}]{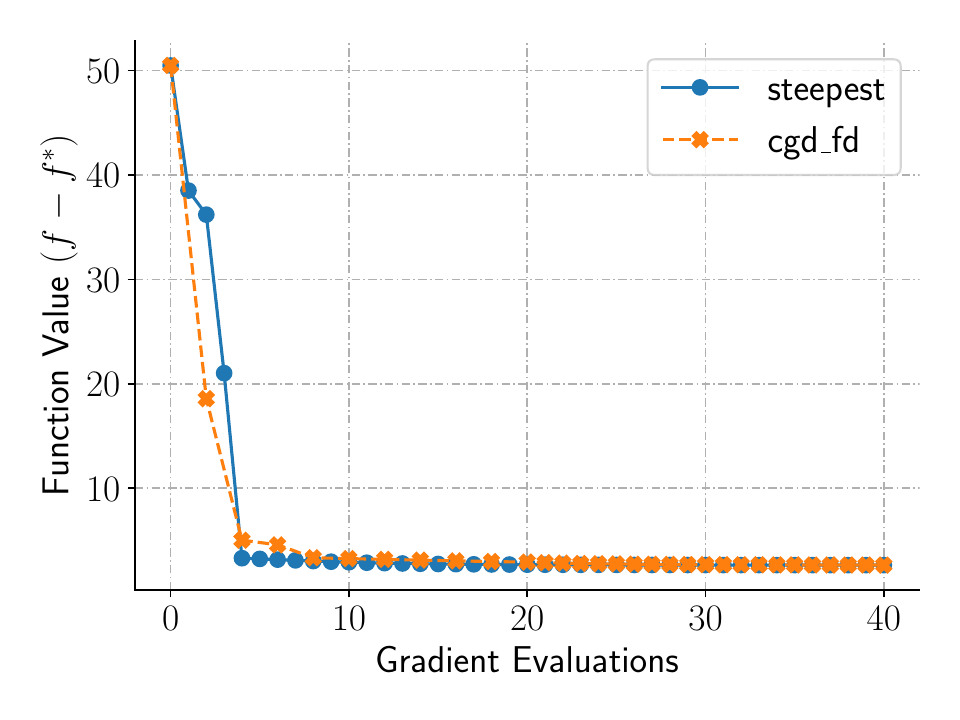}
        \label{fig:levy}
    }\\[1ex]
    \subfloat[Branin function]{%
        \includegraphics[width=0.324\textwidth,
        alt={Line plot showing drop in function value for CGD-FD vs GD on the Branin function. Both methods converge quickly, with CGD-FD slightly ahead.}]{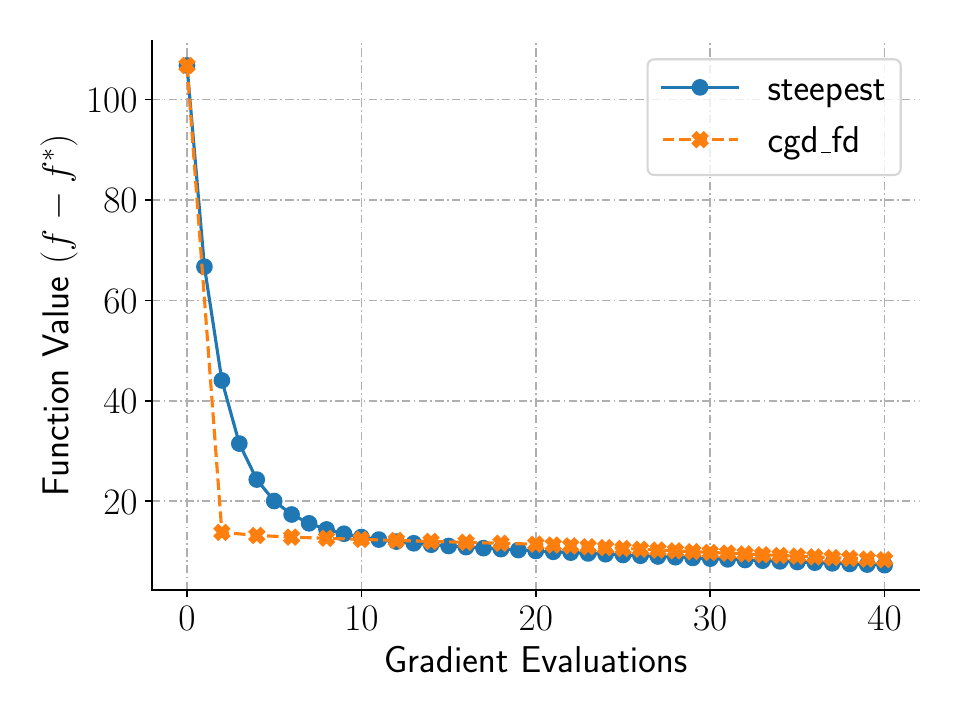}
        \label{fig:branin}
    }\hfill
    \subfloat[Griewank function]{%
        \includegraphics[width=0.324\textwidth,
        alt={Line plot showing drop in function value for CGD-FD vs GD on the Griewank function. CGD-FD reaches a lower value and flattens earlier while GD decreases more gradually.}]{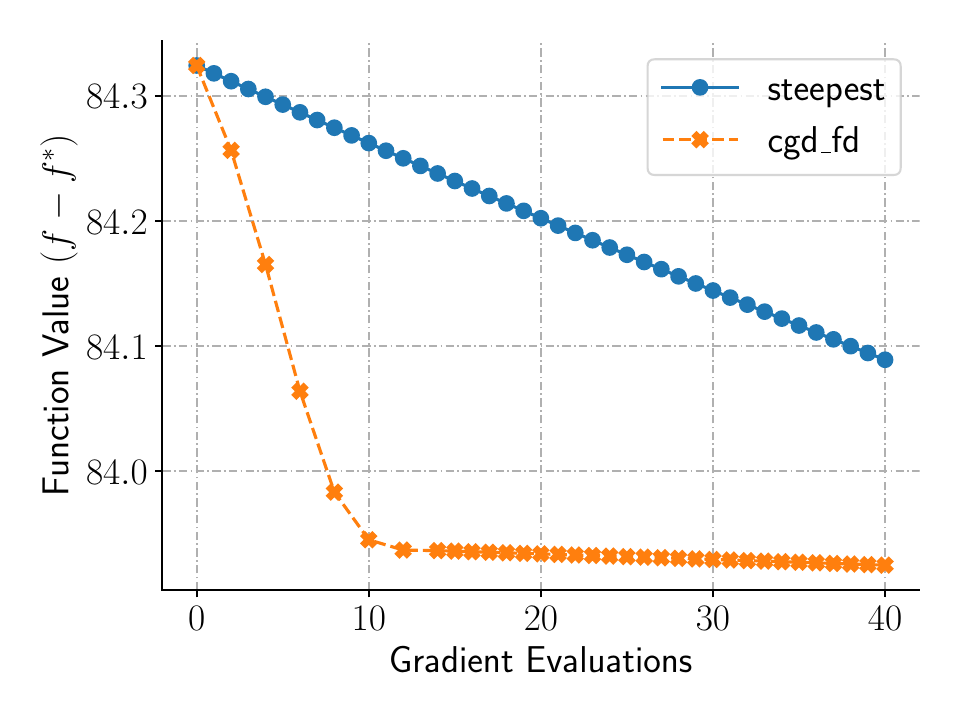}
        \label{fig:griewank}
    }\hfill
    \subfloat[Matyas function]{%
        \includegraphics[width=0.324\textwidth,
        alt={Line plot showing drop in function value for CGD-FD vs GD on the Matyas function. CGD-FD reaches the minimum in fewer steps while GD decreases more gradually.}]{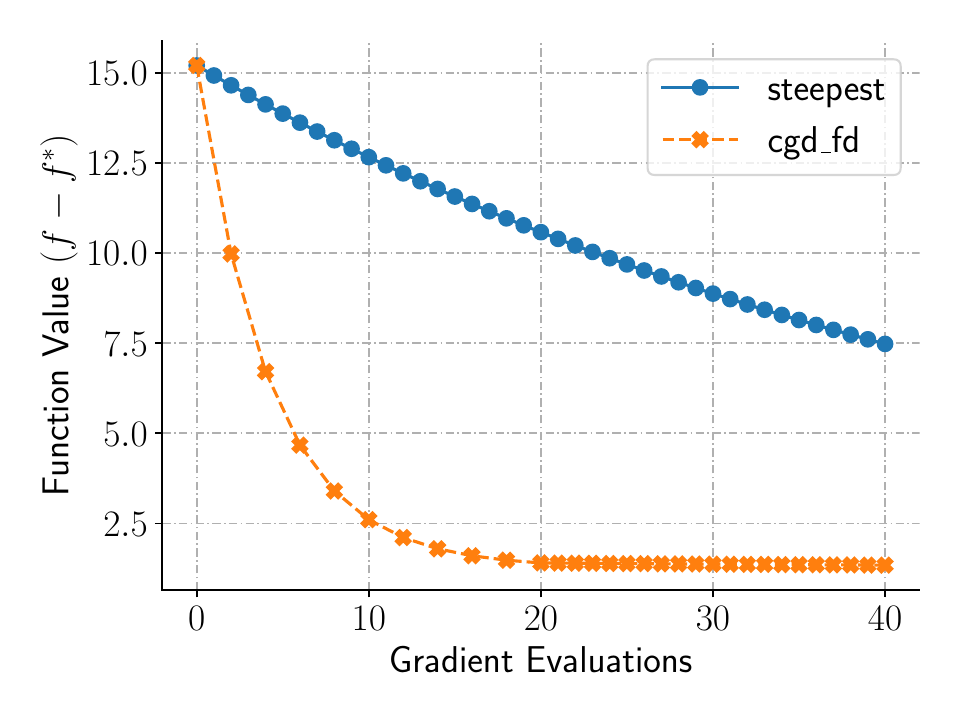}
        \label{fig:matyas}
    }
    \caption[CGD-FD vs GD on six benchmark functions]{Experiments for CGD-FD vs GD: Function value $f(\cdot) - f^*$ vs Gradient Evaluations. \textit{Note}: The $x$-axis of each plot is not iterations but number of gradients evaluated.}
    \label{fig:tab-1}
\end{figure}

Table \ref{tab:tab-1} summarizes our findings and \Figref{fig:tab-1} visualizes function values vs gradient evaluations for the chosen test functions for experiments on CGD-FD.
Observe that in the first step of optimizing Quadratic function (\Figref{fig:qfunc}), CGD-FD has a $97.91\%$ decrease in value compared to that of GD which only had an initial improvement of about $18.89\%$. Similar pattern can also be seen for other functions. This gives us the insight that CGD-FD penalizes the original function well enough to concentrate most of the decrease in value within the first few steps of the trajectory itself.
Also, notice that in Griewank function (\Figref{fig:griewank}) even though the initial step improvement is low ($=0.08\%$), we still see a much rapid decrease in value within the first 10 steps compared to that of the GD trajectory.
Thus, using CGD-FD with appropriate penalization can provide for great boosts in function value well within the starting iterates of the optimization procedure.

\begin{figure}[!htb]
    \centering
    \subfloat[Zakharov function]{%
        \includegraphics[width=0.324\textwidth,
        alt={Line plot showing drop in function value for CGD-QN vs QN methods on the Zakharov function. CGD-QN methods converge faster.}]{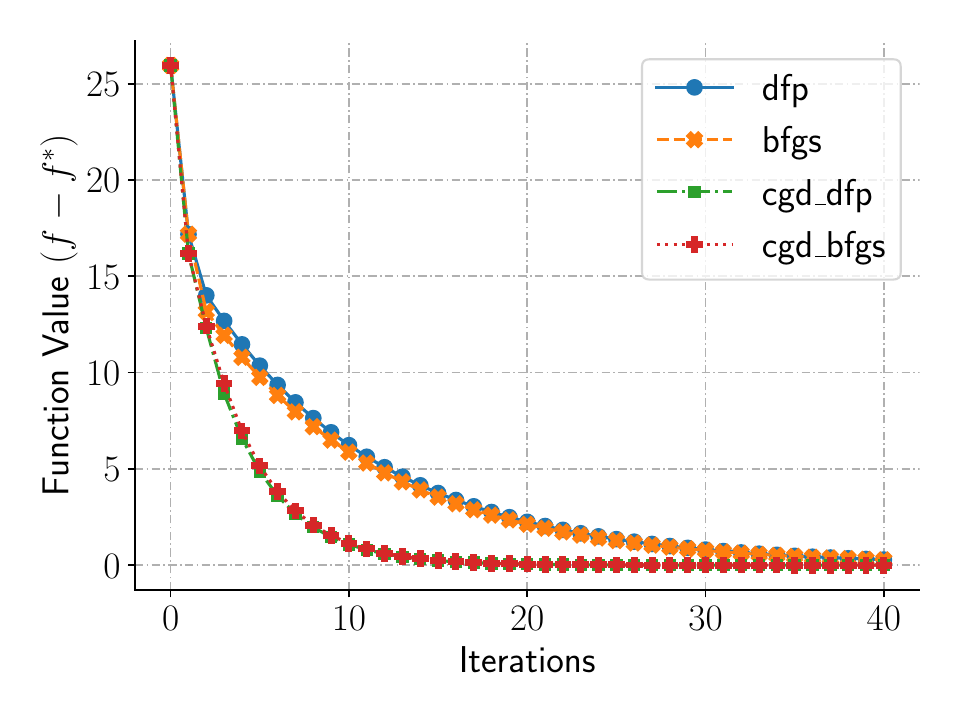}
        \label{fig:zakh}
    }\hfill
    \subfloat[Drop-Wave function]{%
        \includegraphics[width=0.324\textwidth,
        alt={Line plot showing drop in function value for CGD-QN vs QN methods on the Drop-Wave function. Both CGD-QN variants reach lower values earlier, while their vanilla counterparts plateau.}]{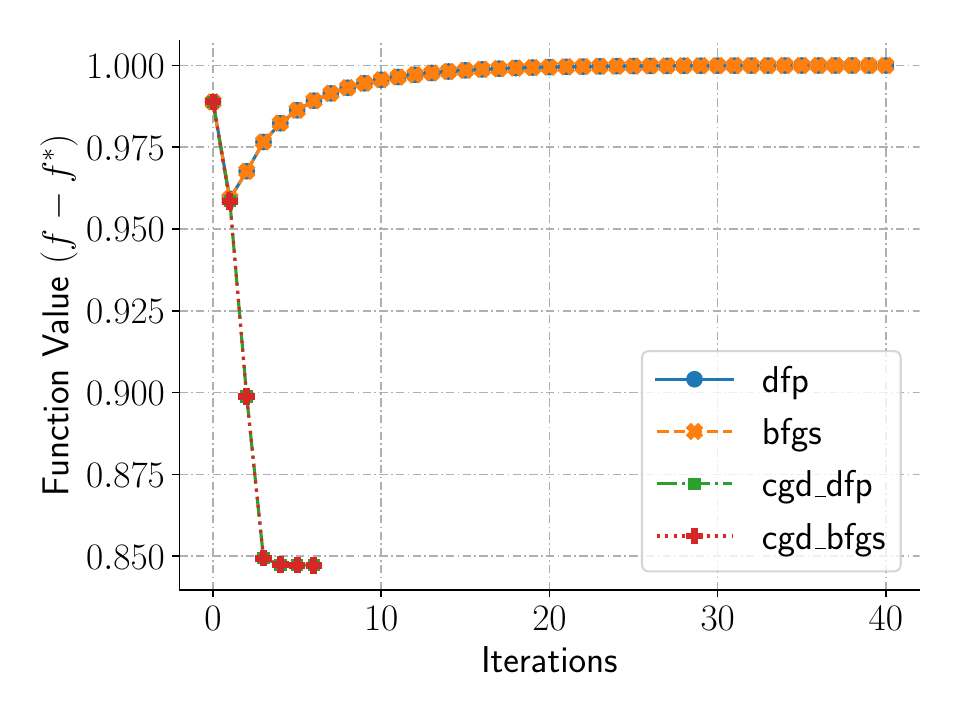}
        \label{fig:dropwave}
    }\hfill
    \subfloat[EggHolder function]{%
        \includegraphics[width=0.324\textwidth,
        alt={Line plot showing drop in function value for CGD-QN vs QN methods on the EggHolder function. Both CGD-QN methods show substantial improvement while their vanilla counterparts remain flat.}]{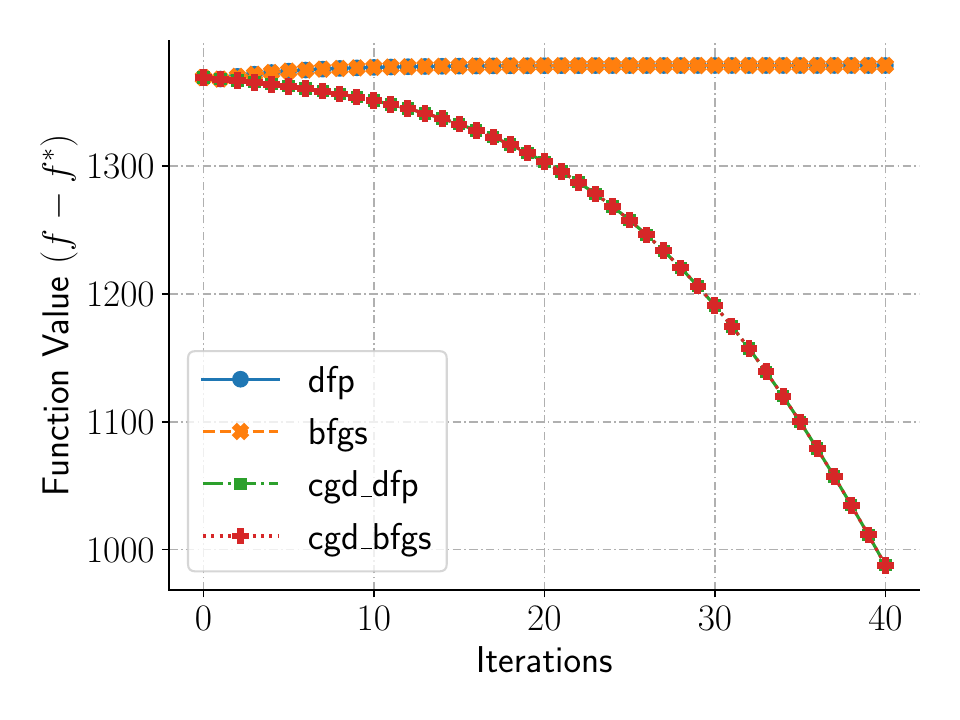}
        \label{fig:eggholder}
    }
    \caption[CGD-QN vs QN methods on three benchmark functions]{Experiments for CGD-QN vs QN methods: Function value $f(\cdot) - f^*$ vs iterations.}
    \label{fig:qn}
\end{figure}

\subsubsection{Experiments for CGD-QN.}
We ran CGD-BFGS and CGD-DFP methods against BFGS and DFP for $T = 40$. The function values vs iterations are visualized in \Figref{fig:qn}. We observe that for most suitable choices of $\lambda$, CGD-DFP and CGD-BFGS exhibit similar behaviour and hence, their trajectories coincide.
Another thing to observe is that in CGD-QN methods, most improvements aren't concentrated in the initial steps since Hessian approximation is still not very great here. Instead the improvements are more gradual and appear over further steps. For example, in Zakharov function (\Figref{fig:zakh}), we can observe that all methods follow the same decrease in function value for the first 1--2 steps and then the CGD-QN methods start to drop more in value. In EggHolder function (\Figref{fig:eggholder}) this effect is most pronounced where the improvements are very gradual but providing with better results than their vanilla counterparts.

\section{Reinterpretation of Explicit Gradient Regularization (EGR)}

Barrett and Dherin~\cite{implicit-grad} proposed \textbf{Explicit Gradient Regularization} (EGR) where the original loss function is regularized with the square of $\normltwo$-norm of the gradient. This regularized objective is then optimized with the intention that the model's parameters will converge to a \textit{flat}-minima and thus, be more generalizable. However, an understanding of why this happens is missing.

\begin{figure}[!htb]
    \centering
    \includegraphics[width=0.45\textwidth,
    alt={Two curves show the original loss landscape and the penalized landscape. Overlaid on these are two optimization trajectories for GD and CGD, respectively. The CGD path explores a broader region and reaches a lower minimum than the GD path, illustrating improved descent behavior by following directions computed from the modified (penalized) loss function.}
    ]{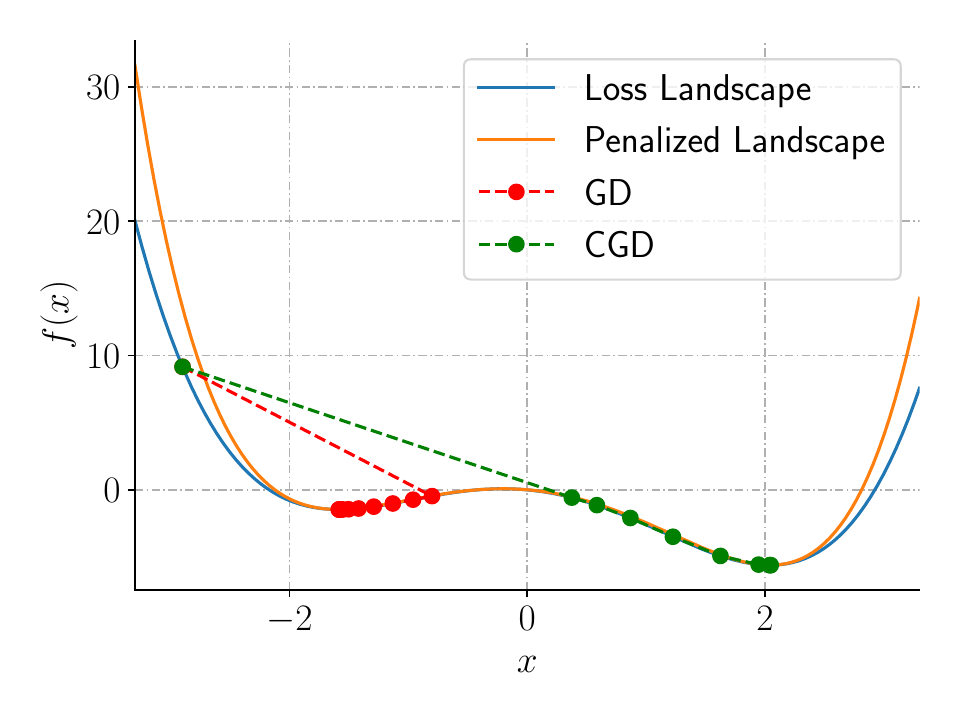}
    \includegraphics[width=0.45\textwidth,
    alt={Two curves showing the original loss landscape and the penalized landscape. The penalized curve slightly deviates from the original, with several regions where local maxima in the loss landscape correspond to local minima in the penalized version.}
    ]{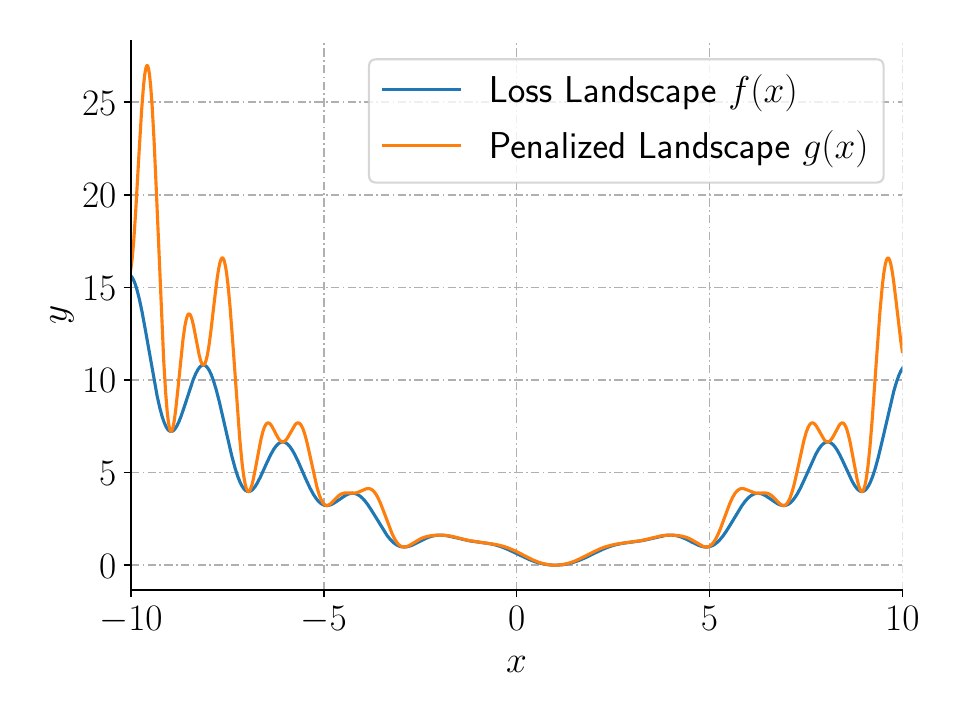}
    \caption{\figleft{} CGD is able to move to a better local minimum by moving along the negative gradients over the modified loss function. \figright{} Local maxima of the original loss function turning into local minima of the penalized loss function.}
    \label{fig:egr}
\end{figure}

We explain EGR through the example visualized in \Figref{fig:egr} \figleft{}. Based on our formulation explained in Section~\ref{sec:cgd}, we can interpret the gradient-regularized function as a steeper version of the original loss function wherein the minima remain same. Due to the gradient penalty, a steep minimum (in the original loss function) would turn steeper while the increase in steepness wouldn't be this high at a flatter minimum. Thus, at a constant step-size, the optimizer will likely overshoot over the sharper minimum due to the extremely high gradient value, while still being able to converge to the less-steeper and preferred flat minimum point. In \Figref{fig:egr} \figleft{} we see how GD gets stuck at a suboptimal point of the function while CGD (GD over the penalized loss function) is able to avoid this point and converges to a better (more optimal) minimum point.

EGR has the drawbacks of having fictitious minima identified in Lemma~\ref{lm:cond_asp}. Since, the loss landscape of neural networks is highly uneven~\cite{DBLP:journals/corr/abs-1712-09913}, it's likely that artificial stationary points are introduced and the optimizer might converge at local-maxima points since their nature changes with regularization. For example, in \Figref{fig:egr} \figright{} we observe artificial stationary points being introduced between local-minima and local-maxima of the original loss function while some local-maxima also turn into local-minima of the penalized loss function. Therefore, one needs to ensure that there are suitable fixes for these scenarios for better performance. Specifically, the direction along of the penalized loss function, $-\nabla g(\cdot)$ should be a descent direction and that we shouldn't stop at points where $\nabla g(\cdot) = \vzero$ but $\nabla f(\cdot) \not = \vzero$ as discussed in Lemma~\ref{lm:cond_asp}.

Another downfall with EGR is that it requires hessian evaluation for each mini-batch while training. This might become a costly operation for bigger models and hence some approximation of the hessian should be used. Another thing we observe from the experimental results for CGD-FD, is that the improvement through optimizing the regularized function is mostly only during the initial steps. Hence, one should only use EGR for some initial steps to reach a good-enough starting point while not exhausting the budget for gradient evaluations.

\section{Conclusions and Future Work}
 In this work, we considered a new line search method that penalizes the norm of the gradient and provides iterates that have lower gradient norm compared to vanilla gradient descent. We identify properties of this algorithm, provide variants that do not require the Hessian and illustrate connections to the widely popular explicit gradient regularization literature.
 
 There are several future directions arising from this work. We would like to investigate the role of different penalty functions $h(\cdot)$ in greater detail.
 Another future direction is using a combination of negative and positive $\{\lambda_k\}$ values in our iterates. Particularly, negative $\{\lambda_k\}$ values might help in descending faster at points where the Hessian is negative definite.
 We would also like to explore applications of our method in more diverse settings like reinforcement learning (RL) and Bayesian optimization (BO).
 In particular, for RL, we can develop penalized variants of existing Policy Gradient Methods~\cite{Williams_1992,pgm_rl,pmlr-v37-schulman15}, which can be evaluated to assess the effectiveness of gradient regularization in RL settings.
 Similarly, for BO tasks, we can construct acquisition functions involving a gradient penalty, in a manner similar to \cite{practical_fobo}, which might improve performance.
 We hope this work sparks further discussion on gradient regularization helping in neural network generalization.

\begin{credits}

\subsubsection{\ackname} 
We thank IHub-Data, IIIT Hyderabad for the research fellowship supporting this work. This work was also supported by the SERB MATRICS project (Grant No. MTR/2023/000042) from the Science and Engineering Research Board (SERB).

\subsubsection{\discintname}
The authors have no competing interests to declare that are relevant to the content of this article.
\end{credits}

\bibliographystyle{splncs04}
\bibliography{references}

\end{document}